\documentclass[reqno]{amsart}

\usepackage[english]{babel}
\usepackage{csquotes}
\usepackage[style=numeric-comp,firstinits=true,backend=bibtex,maxnames=6,
]{biblatex}
\bibliography{Diagonal_Differential_Operators.bib}

\DeclareFieldFormat[article,book,incollection,phdthesis,misc]{volume}{\mkbibbold{#1}}
\DeclareFieldFormat[article]{publisher}{\textnormal{#1}}
\DeclareFieldFormat[article,book,phdthesis]{title}{\mkbibemph{#1}}
\usepackage{amsmath}
\usepackage{amsthm}
\usepackage{amsfonts}
\usepackage{amssymb}
\usepackage{mathrsfs}
\usepackage{enumerate}
\usepackage[usenames,dvipsnames]{xcolor}
\usepackage{graphicx}
\usepackage{tikz,caption}
\usetikzlibrary{arrows,shapes,trees}
\usepackage{verbatim}
\usepackage{url}
\usepackage[hyperindex,breaklinks]{hyperref}
\hypersetup{linkcolor=blue, citecolor=red, colorlinks=true}
\usepackage{array}
\usepackage[a4paper,margin=1in,footskip=0.25in]{geometry}

\newtheorem{thm}{Theorem}

\newtheorem{cor}[thm]{Corollary}

\theoremstyle{definition}
\newtheorem{defn}[thm]{Definition}
\newtheorem{rmk}[thm]{Remark}
\newtheorem{exmp}[thm]{Example}

\newtheorem{ques}[thm]{Question}

\providecommand{\ds}[0]{\displaystyle}

\providecommand{\R}[0]{\mathbb{R}}

\providecommand{\N}[0]{\mathbb{N}}

\newcommand{\bdf}{\begin{defn}}
\newcommand{\edf}{\end{defn}}

\newcommand{\brm}{\begin{rmk}}
\newcommand{\erm}{\end{rmk}}

\newcommand{\beq}{\begin{equation}}
\newcommand{\eeq}{\end{equation}}

\newcommand{\barr}[1]{\begin{array}{#1}}
\newcommand{\earr}{\end{array}}

\newcommand{\lr}[1]{\left(#1\right)}

\numberwithin{equation}{section}
\numberwithin{figure}{section}

\author{Robert D. Bates}
\date{\today}
\title{Diagonal Differential Operators}
\email{rdbates@math.hawaii.edu}

\begin{document}
\setcounter{page}{1}
\pagenumbering{roman}

\begin{abstract}
We explore differential operators, $T$, that diagonalize on a simple basis, $\{B_n(x)\}_{n=0}^\infty$, with respect to some sequence of real numbers, $\{a_n\}_{n=0}^\infty$, and sequence of polynomials, $\{Q_k(x)\}_{k=0}^\infty$, as in $ T[B_n(x)]:=\left(\sum_{k=0}^\infty Q_k(x) D^k\right)B_n(x)=a_n B_n(x)$ for every $n\in\N_0$. We discover new relationships between the sequence, $\{Q_k(x)\}_{k=0}^\infty$, and the sequence, $\{a_n\}_{n=0}^\infty$. We find new relationships between polynomial interpolated eigenvalues and the sequence, $\{\deg(Q_k(x))\}_{k=0}^\infty$. 
\end{abstract}

 \maketitle
\thispagestyle{empty}




\section{Introduction}

\setcounter{page}{1}
\pagenumbering{arabic}
 
A quintessential unsolved problem in the theory of differential operators, 
\beq\label{eq:diffrep}
T:= \sum_{k=0}^\infty Q_k(x)D^k,\ \text{where}\ D:=\frac{d}{dx},
\eeq
is to characterize the properties of the polynomials, $\{Q_k(x)\}_{k=0}^\infty$, so that $T$ preserves the class of polynomials with only real zeros (the Laguerre-P\'olya Class) (Definition \ref{defhp}), (see, for example, \cite{BB10,Cha11,BDFU11,BB09,BY13} and the references therein). However, the aforementioned problem is still open. We restrict our attention to differential operators that are diagonal with respect to an eigenvector sequence of polynomials, $\{B_n(x)\}_{n=0}^\infty$ ($\deg(B_n(x))=n$ for each $n\in\N_0$, $B_0(x)\not\equiv 0$), and eigenvalue sequence of real numbers, $\{a_n\}_{n=0}^\infty$, such that
\beq
T[B_n(x)]:=\left(\sum_{k=0}^\infty Q_k(x)D^k\right)B_n(x)=a_nB_n(x),\ \ \ \ \ n\in\N_0.  
\eeq
Our main result characterizes all possible eigenvalues of a diagonal differential operator (Theorem \ref{thm:dia2}), a generalization of A. Piotrowski \cite[Proposition 33, p. 35]{Pio07}. We show, in particular, that if the sequence $\{a_n\}_{n=0}^\infty$ cannot be interpolated by a polynomial (Definition \ref{def:interp}), then $T$ must be an infinite order (Definition \ref{order}) differential operator (Corollary \ref{cor:intercon}). We generalize a result of L. Miranian \cite{Mir05} and show for some eigenvector sequences, $\{B_n(x)\}_{n=0}^\infty$, that $T$ is a finite order diagonal differential operator if and only if the sequence, $\{a_n\}_{n=0}^\infty$, can be interpolated by a polynomial (Theorem \ref{finiffint}). Moreover, if $T$ is hyperbolicity preserving with positive increasing eigenvalues, then $\deg(Q_k(x))=k$ for every $k\in\N_0$ up to the degree of the polynomial that interpolates $\{a_n\}_{n=0}^\infty$. If the sequence, $\{a_n\}_{n=0}^\infty$, cannot be interpolated by a polynomial, then $\deg(Q_k(x))=k$ for all $k\in\N_0$ (Theorem \ref{final}). 

Our investigations are divided into two parts. First, we explore the properties of eigenvalues and eigenvectors in a diagonal differential operator. Second, we discuss these results in conjunction with hyperbolicity preserving operators and multiplier sequences.

\hspace{.5in}

\section{Diagonal Differential Operators}\label{section1}

\begin{defn}\label{diadiffop}
Let $T:\R[x]\to\R[x]$ be a linear operator. We say that $T$ is a \textit{diagonal differential operator}, if there exists a sequence of real numbers $\{a_n\}_{n=0}^\infty$, and a sequence of real polynomials, $\{B_n(x)\}_{n=0}^\infty$ ($\deg(B_n(x))=n$ for each $n\in\N_0$, $B_0(x)\not\equiv 0$), such that for every $n\in\N_0$, $T[B_n(x)]=a_n B_n(x)$. For clarity, we note that a diagonal differential operator, $T:\R[x]\to\R[x]$, is a linear operator with the additional assumption that there is an eigenvector of degree $n$, for each $n\in\N_0$. 
\end{defn}

\begin{defn}\label{def:interp}
Let $\{a_n\}_{n=0}^\infty$ be a sequence of real numbers. If there is a polynomial, $p(x)$, such that $a_n=p(n)$ for every $n\in\N_0$, then we say that $\{a_n\}_{n=0}^\infty$ is \textit{interpolated by a polynomial}. If no such polynomial can be found, then $\{a_n\}_{n=0}^\infty$ \textit{cannot be interpolated by a polynomial}. 
\end{defn}

\begin{defn}\label{order}
Let $T:\R[x]\to\R[x]$ be a linear operator of the form, 
\beq\label{eq:diffrep2}
T:=\sum_{k=0}^\infty Q_k(x)D^k, 
\eeq
where $\{Q_k(x)\}_{k=0}^\infty$ is some sequence of real polynomials. We say $T$ is a \textit{finite order} differential operator of order $n$, if $Q_n(x)\not\equiv 0$ and $Q_k(x)\equiv 0$ for $k>n$. Similarly, we say that $T$ is an \textit{infinite order} differential operator if $Q_k(x)\not\equiv 0$ for infinitely many $k\in\N_0$. 
\end{defn}

We begin this section with a slight extension of Peetre's 1959 result on linear operator representation \cite{Pee59} (see also \cite[p. 32]{Pio07}), which demonstrates that every linear operator, from $\R[x]$ to $\R[x]$, can be represented in the form of \eqref{eq:diffrep2}.  

\begin{thm}\label{piotr-thm1}
If $T:\R[x]\to\R[x]$ is any linear operator, then there is a unique sequence of polynomials, $\{Q_k(x)\}_{k=0}^\infty\subset\R[x]$, such that 
\begin{equation} \label{eq:linear-operator}
T=\sum_{k=0}^\infty Q_k(x) D^k,\ \mbox{where } D:=\frac{d}{dx}.
\end{equation}
Furthermore, given any sequence of polynomials, $\{B_n(x)\}_{n=0}^\infty$ \textup{(}$\deg(B_n(x))=n$ for each $n\in\N_0$, $B_0(x)\not\equiv 0$\textup{)}, then for each $n\in\N_0$, 
\beq\label{equ13}
Q_n(x)=\frac{1}{B_n^{(n)}}\left(T[B_n(x)]-\sum_{k=0}^{n-1} Q_k(x)B_n^{(k)}(x)\right). 
\eeq
\end{thm}

\begin{proof}
We begin by defining the sequence of polynomials, $\{Q_n(x)\}_{n=0}^\infty$, by the recursive formula of \eqref{equ13}. Hence, solving for $T[B_n(x)]$ in \eqref{equ13}, yields, 
\beq
T[B_n(x)]=\left(\sum_{k=0}^n Q_k(x) D^k \right) B_n(x), 
\eeq
for each $n\in\N_0$. Thus, define the operator, 
\beq
\tilde{T}:=\sum_{k=0}^\infty Q_k(x) D^k, 
\eeq
and note that $T[B_n(x)]=\tilde{T}[B_n(x)]$ for every $n\in\N_0$. By linearity, we conclude that $T[p(x)]=\tilde{T}[p(x)]$ for every $p(x)\in\R[x]$. Hence, $T$ can be represented by the differential operator \eqref{eq:linear-operator} where the $Q_k$'s are given by the recursive formula in \eqref{equ13}. We now establish uniqueness. Suppose there are sequences of real polynomials, $\{Q_k(x)\}_{k=0}^\infty$ and $\{\tilde{Q}_k(x)\}_{k=0}^\infty$, such that 
\beq
T=\sum_{k=0}^\infty Q_k(x)D^k\ \ \ \text{and}\ \ \ T=\sum_{k=0}^\infty \tilde{Q}_k(x)D^k. 
\eeq
Thus, 
\beq
\sum_{k=0}^{\deg(p(x))} Q_k(x)p^{(k)}(x) = \sum_{k=0}^{\deg(p(x))} \tilde{Q}_k(x) p^{(k)}(x), 
\eeq
for every $p(x)\in\R[x]$. For $p(x)= 1$, this gives $Q_0(x)=\tilde{Q}_0(x)$. For $p(x)=x$, this gives $Q_0(x)x+Q_1(x) = \tilde{Q}_0(x)x+\tilde{Q}_1(x)$ and hence, $Q_1(x)=\tilde{Q}_1(x)$. We continue by induction and conclude that $Q_n(x)=\tilde{Q}_n(x)$ for every $n\in\N_0$. 
\end{proof}

The next two theorems include some combinatorial facts that will be useful. 

\begin{thm}[{\cite[p. 49]{Rio79}}]\label{comb1}
If $\{a_n\}_{n=0}^\infty$ is a sequence of real numbers and we define for each $n\in\N_0$, 
\beq
c_n=\sum_{k=0}^n \binom{n}{k} a_k, 
\eeq
then for each $n\in\N_0$, 
\beq
a_n=\sum_{k=0}^n \binom{n}{k} c_k (-1)^{n-k}.  
\eeq
\end{thm}

\begin{thm}\label{comb2}
Let $p(x)$ be a real polynomial and for each $n\in\N_0$ set
\beq
a_n=\sum_{k=0}^n \binom{n}{k} p(k) (-1)^{n-k}.  
\eeq
If $n>\deg(p(x))$, then $a_n=0$. If $n=\deg(p(x))$, then $a_n=n!p^{(n)}\not=0$.
\end{thm}

\begin{proof}
It is well known that
\beq
\sum_{k=0}^n \binom{n}{k}x^{k}=(1+x)^n,
\eeq 
thus, 
\beq\label{eq:thing7}
\sum_{k=0}^n \binom{n}{k}k^m x^{k}=(xD)^m(1+x)^n,\ \text{where}\ D:=\frac{d}{dx}. 
\eeq
In equation \eqref{eq:thing7}, each differentiation reduces by one the multiplicity of the zero at $-1$. Thus we have,  
\beq\label{gould-calc}
\left.\sum_{k=0}^n \binom{n}{k}k^m (-1)^{n-k}= (-1)^n(xD)^m (1+x)^n\right|_{x=-1}=\left\{\begin{matrix}0 & 0\le m<n \\ n! & m=n \end{matrix}\right..
\eeq
Calculation \eqref{gould-calc} is also found in \cite[Equation 1.13, p. 2]{Gou72}. Thus, for polynomial, $p(x):=c_0+c_1x+\cdots+c_nx^n$, $\deg(p(x))\le n$, we have, 
\begin{align}
\sum_{k=0}^n \binom{n}{k}p(k) (-1)^{n-k} 
&= \sum_{k=0}^n \binom{n}{k}\left(c_0+c_1k+\cdots+c_nk^n\right) (-1)^{n-k} \\ 
&= \left(\sum_{k=0}^n \binom{n}{k} (-1)^{n-k}\right)c_0 +\cdots+ \left(\sum_{k=0}^n \binom{n}{k} k^n (-1)^{n-k}\right)c_n=n!c_n.
\end{align}
The result now follows. 
\end{proof}

\begin{thm}\label{thm:dia2}
If $T$ is a diagonal differential operator, 
\beq
T[B_n(x)]:=\left(\sum_{k=0}^\infty Q_k(x)D^k\right)B_n(x)=a_nB_n(x),\ \ \ \ \ n\in\N_0,
\eeq
then 
\beq\label{form1}
a_n=\sum_{k=0}^n \binom{n}{k} Q_k^{(k)}
\eeq
and 
\beq\label{form2}
Q_n^{(n)}=\sum_{k=0}^n \binom{n}{k} a_k (-1)^{n-k}. 
\eeq
\end{thm}

\begin{proof}
Notice that for each $n\in\N_0$, $\deg\left(T[B_n(x)]\right)=\deg(a_n B_n(x))\le n$. Hence, the recursive formula of Theorem \ref{piotr-thm1}, establishes by induction that $\deg(Q_k(x))\le k$ for each $k\in\N_0$. Thus, by assumption, since 
\beq
\left(\sum_{k=0}^\infty Q_k(x)D^k\right)B_n(x)=a_n B_n(x),
\eeq
then differentiating both sides $n$ times yields, 
\beq
\sum_{k=0}^n \sum_{j=0}^n \binom{n}{j} Q_k^{(j)}(0)B_n^{(k+n-j)}(0) = a_n B_n^{(n)}. 
\eeq
Hence, since $B_n(x)$ is of degree $n$ and $\deg(Q_k(x))\le k$, we have,
\beq
\sum_{k=0}^n \binom{n}{k} Q_k^{(k)}=a_n.
\eeq
Equation \eqref{form2} follows from Theorem \ref{comb1}. 
\end{proof}

As we will find, Theorem \ref{thm:dia2} will lead to a number of interesting observations concerning diagonal differential operators. In many places, formula \eqref{form1} is typically assumed as a condition of study (for example \cite{Kra38}) or additional hypotheses are imposed to obtain \eqref{form1} (see \cite{KS66, Zhe14}). In the orthogonal case, this calculation is also found in \cite{SHT13, Sze59}; in particular, in \cite[Theorem 4.2.2, p. 61]{Sze59} somewhat more difficult abstract methods are used to derive \eqref{form1} for the Jacobi polynomials. Equation \eqref{form1} also appears to partially solve a problem of Fisk \cite[Question 115, p. 731]{Fis06}. We also note that Theorem \ref{thm:dia2} is a generalization of A. Piotrowski's observation \cite[Proposition 33, p. 35]{Pio07}. 

Our first interesting result provides necessary and sufficient conditions for the sequence of eigenvalues of a diagonal differential operator to be interpolated by a polynomial. 

\begin{thm}\label{thm:degqk}
Suppose $T$ is a diagonal differential operator, 
\beq\label{eq:thing15}
T[B_n(x)]:=\left(\sum_{k=0}^\infty Q_k(x)D^k\right)B_n(x)=a_nB_n(x),\ \ \ \ \ n\in\N_0.
\eeq
Then $\{a_n\}_{n=0}^\infty$ can be interpolated by a polynomial of degree $m$ if and only if $\deg(Q_m(x))=m$ and $\deg(Q_k(x))<k$ for $k>m$. Moreover, $\{a_n\}_{n=0}^\infty$ cannot be interpolated by a polynomial if and only if $\deg(Q_k(x))=k$ for infinitely many $k\in\N_0$. 
\end{thm}

\begin{proof}
As was noted in the proof of Theorem \ref{thm:dia2}, $\deg(Q_k(x))\le k$ for each $k\in\N_0$. In addition, by Theorem \ref{thm:dia2}, 
\beq
Q_n^{(n)}=\sum_{k=0}^n \binom{n}{k} a_k (-1)^{n-k}. 
\eeq
Thus, if $\{a_n\}_{n=0}^\infty$ can be interpolated by a polynomial of degree $m$, $p(x)$, (i.e., $p(n)=a_n$ for all $n\in\N_0$, see Definition \ref{def:interp}), then by Theorem \ref{comb2}, $Q_m^{(m)}\not=0$ and $Q_k^{(k)}=0$ for $k>m$; that is, $\deg(Q_m(x))=m$ and $\deg(Q_k(x))<k$ for all $k>m$. Conversely, if there is $m\in\N_0$ such that $\deg(Q_m(x))=m$ and $\deg(Q_k(x))<k$ for all $k>m$ (that is, $Q_m^{(m)}\not=0$ and $Q_k^{(k)}=0$ for k$>m$), then, by Theorem \ref{thm:dia2}, $\{a_n\}_{n=0}^\infty$ can be interpolated by the polynomial of degree $m$, 
\beq
p(x):=\sum_{k=0}^\infty \binom{x}{k}Q_k^{(k)}, 
\eeq
where for each $k\in\N_0$, $\ds \binom{x}{k}=\frac{1}{k!}x(x-1)\cdots(x-(k-1))$ is a polynomial of degree $k$. 
\end{proof}

\begin{cor}\label{cor:intercon}
Suppose $T$ is a diagonal differential operator, 
\beq
T[B_n(x)]:=\left(\sum_{k=0}^\infty Q_k(x)D^k\right)B_n(x)=a_nB_n(x),\ \ \ \ \ n\in\N_0.
\eeq
Then each of the following statements hold. 
\begin{enumerate}
\item[\textup{(1)}] If $T$ is of finite order, then $\{a_n\}_{n=0}^\infty$ can be interpolated by a polynomial. 
\item[\textup{(2)}] If $\{a_n\}_{n=0}^\infty$ cannot be interpolated by a polynomial, then $T$ is of infinite order. 
\item[\textup{(3)}] If $\{a_n\}_{n=0}^\infty$ is a non-constant sequence with a bounded sub-sequence, then $T$ is of infinite order. 
\item[\textup{(4)}] If $\{a_n\}_{n=0}^\infty$ is either a non-constant, non-negative, decreasing sequence or a non-constant, non-positive, increasing sequence, then $T$ is of infinite order. 
\item[\textup{(5)}] If $\{a_n\}_{n=0}^\infty$ is a non-constant alternating sequence, then $T$ is of infinite order.  
\end{enumerate}
\end{cor}

\begin{exmp}\label{ex:classicinfinite}
Consider the diagonal differential operator, 
\beq
T[x^n]=\frac{1}{n!}x^n. 
\eeq
Operator $T$ has the following differential representation (see Theorem \ref{piotr-thm1}), 
\begin{align}
T=&\sum_{n=0}^\infty\frac{1}{n!}\left(\sum_{k=0}^n \binom{n}{k}\frac{1}{k!}(-1)^{n-k}\right)x^nD^n=1-\frac{1}{2}x^2D^2+\frac{2}{3}x^3D^3-\frac{5}{8}x^4D^4+\cdots. 
\end{align}
Corollary \ref{cor:intercon} item (3) implies that $T$ is an infinite order differential operator. 
\end{exmp}

\begin{exmp}\label{ex:fininfherm}
Consider the following Hermite diagonal differential operators, 
\beq
T[H_n(x)]=nH_n(x)\ \ \ \text{and}\ \ \ \tilde{T}[H_n(x)]=(-1)^n nH_n(x), 
\eeq
where $H_n(x)$ denotes the $n^{\text{th}}$ Hermite polynomial (see \cite[p. 187]{Rai60}). We provide the first few polynomial coefficients from Theorem \ref{piotr-thm1}, 
\beq
T=xD-\frac{1}{2}D^2, 
\eeq
and
\beq
\tilde{T}=-xD+\left(2x^2-\frac{1}{2}\right)D^2+\left(-2x^3+x\right)D^3+\cdots. 
\eeq
We see that $T$ is a finite order differential operator while Corollary \ref{cor:intercon} item (5) shows that $\tilde{T}$ is an infinite order differential operator. 
\end{exmp}

\begin{exmp}\label{ex:leginfinite}
While the notions of ``polynomial interpolation'' and ``finite order'' seem intimately connected (Theorem \ref{thm:degqk}), there exist operators that are of infinite order and have eigenvalues that can be interpolated by a polynomial. Consider the following Legendre diagonal differential operator, 
\beq
T[P_n(x)]=nP_n(x),  
\eeq
where $P_n(x)$ denotes the $n^{\text{th}}$ Legendre polynomial (see \cite[p. 157]{Rai60}). Using the Legendre differential equation, we have, 
\beq
(T^2+T)[P_n(x)]=(n^2+n)P_n(x)=((x^2-1)D^2+2xD)P_n(x).  
\eeq
If $T$ is of finite order $\left(\ds T:=\sum_{k=0}^n Q_k(x)D^k\right)$, then $T^2+T$ is also of finite order ($T^2+T=(Q_n(x))^2 D^{2n}+\cdots$). Hence, we conclude that $T$ must be of order one, since $T^2+T$ is of order two; i.e., $T=Q_1(x)D+Q_0(x)$, where $Q_1(x)$ and $Q_0(x)$ are polynomials. Moreover, the polynomial coefficients of $D^2$ are equal and therefore $Q_1(x)^2=x^2-1$. But this is impossible for a polynomial. Hence, $T$ cannot be of finite order. Thus, $T$ is an infinite order diagonal differential operator with eigenvalues that can be interpolated by a polynomial. For clarity, by Theorem \ref{piotr-thm1}, we provide the first few $Q_k$'s of operator $T$, 
\beq
T=(x)D+\left(-\frac{1}{3}\right)D^2+\left(\frac{2}{15}x\right)D^3+\left(-\frac{4}{105}x^2-\frac{1}{105}\right)D^4+\cdots
\eeq
Note that although $T$ is of infinite order, Theorem \ref{thm:degqk} still holds, in that $\deg(Q_k(x))<k$ for $k>2$. 
\end{exmp}

We now generalize when polynomial interpolation will correspond to a finite order differential operator. For example, consider the Hermite or Laguerre basis of L. Miranian \cite{Mir05}. 

\begin{thm}\label{finiffint}
Suppose for the basis, $\{B_n(x)\}_{n=0}^\infty$, there is a finite order differential operator, $W$, such that $W[B_n(x)]=nB_n(x)$. Now suppose $T$ is any diagonal differential operator such that,
\beq
T[B_n(x)]:=\left(\sum_{k=0}^\infty Q_k(x)D^k\right)B_n(x)=a_nB_n(x),\ \ \ \ \ n\in\N_0. 
\eeq
Then $\{a_n\}_{n=0}^\infty$ can be interpolated by a polynomial if and only if $T$ is a finite order differential operator. 
\end{thm} 

\begin{proof}
If $T$ is a finite order diagonal differential operator, then, by Corollary \ref{cor:intercon} item (1), $\{a_n\}_{n=0}^\infty$ can be interpolated by a polynomial. Conversely, suppose that $\{a_n\}_{n=0}^\infty$ can be interpolated by a polynomial, $p(x)$. Using the operator, $W$, from our assumption, we observe that $p(W)B_n(x)=p(n)B_n(x)=a_n B_n(x)$. Hence, by uniqueness in Theorem \ref{piotr-thm1}, $T=p(W)$. Thus, $T$ is a finite order diagonal differential operator. 
\end{proof}

\begin{ques}
In relation to the results of L. Miranian \cite{Mir05} and Theorem \ref{finiffint}, we might wonder if a more general statement is possible. Suppose $\{B_n(x)\}_{n=0}^\infty$ is a simple basis of polynomials, and $W$ is a finite order differential operator of ``smallest order'' that diagonalizes $\{B_n(x)\}_{n=0}^\infty$, as in,
\beq
W[B_n(x)]=a_nB_n(x).  
\eeq 
By ``smallest order'' we mean that if $U$ is any other operator that diagonalizes $\{B_n(x)\}_{n=0}^\infty$, then the order of $W$ is smaller than the order of $U$. This leads to the following question. If $T$ is any other finite order differential operator that diagonalizes $\{B_n(x)\}_{n=0}^\infty$, 
\beq
T[B_n(x)]=c_nB_n(x),
\eeq
then must there exist a polynomial, $p(x)$, such that, 
\beq
p(W)B_n(x)=T[B_n(x)]=c_nB_n(x)=p(a_n)B_n(x),
\eeq
for every $n\in\N_0$? 
\end{ques}

We now begin work on the uniqueness of eigenvalues and eigenvectors in a diagonal differential operators. 

\begin{thm}\label{unique1}
Suppose $T$ is a diagonal differential operator with respect to $\{a_n\}_{n=0}^\infty$ and $\{B_n(x)\}_{n=0}^\infty$ and with respect to $\{\tilde{a}_n\}$ and $\{\tilde{B}_n(x)\}_{n=0}^\infty$; that is, 
\beq
T[B_n(x)]=a_nB_n(x)\ \ \ \text{and}\ \ \ T[\tilde{B}_n(x)]=\tilde{a}_n\tilde{B}_n(x). 
\eeq 
Then $a_n=\tilde{a}_n$ for all $n\in\N_0$. 
\end{thm}

\begin{proof}
Use Theorem \ref{thm:dia2} after noting that formula \eqref{form1} is independent of basis. 
\end{proof}

Upon stating the above theorem, we immediately ask if the $B_n$'s in a diagonal differential operator are also unique. Simple examples demonstrate that this is not the case. In particular, for any sequence $\{\beta_n\}_{n=0}^\infty\subseteq\R$ such that $\tilde{B}_n(x):=\beta_n B_n(x)$, where $T[B_n(x)]=a_nB_n(x)$, then certainly, by linearity, $T[\tilde{B}_n(x)]=a_n\tilde{B}_n(x)$; that is, any diagonal differential operator, $T$, will diagonalize on multiples of it's diagonalization. However, under additional restrictions on the eigenvalues, then we can show that the basis chosen for diagonalization is unique up to the constant multiplies noted above. 

\begin{thm}[{\cite[H. Krall and I. Scheffer]{KS66}}]\label{unique}
Suppose $T$ is a diagonal differential operator with respect to $\{a_n\}_{n=0}^\infty$ and $\{B_n(x)\}_{n=0}^\infty$ and with respect to $\{a_n\}_{n=0}^\infty$ and $\{\tilde{B}_n\}_{n=0}^\infty$,
\beq
T[B_n(x)]=a_nB_n(x)\ \ \ \text{and}\ \ \ T[\tilde{B}_n(x)]=a_n\tilde{B}_n(x),\ \ \ \ \ n\in\N_0.
\eeq
For a fixed $m$, suppose $a_m\not=a_k$ for all $0\le k < m$. Then there is $\beta\in\R$, $\beta\not=0$, such that 
\beq
B_m(x)=\beta \tilde{B}_m(x). 
\eeq
\end{thm}

\begin{proof}
Since $\{\tilde{B}_n(x)\}_{n=0}^\infty$ is a simple basis, $B_m(x)=\beta_m \tilde{B}_m(x)+\beta_{m-1} \tilde{B}_{m-1}(x)+\cdots+\beta_0 \tilde{B}_0(x)$, $\beta_m\not=0$. We now apply $T$ to $B_m(x)$ and calculate in two different ways,
\beq
T[B_m(x)]= a_mB_m(x) = a_m\beta_m \tilde{B}_m(x)+a_m\beta_{m-1} \tilde{B}_{m-1}(x)+\cdots+a_m\beta_0 \tilde{B}_0(x),
\eeq
and
\beq
T[B_m(x)]=T[\beta_m\tilde{B}_m(x)+\cdots+\beta_0\tilde{B}_0(x)]=a_m\beta_m\tilde{B}_m(x)+a_{m-1}\beta_{m-1}\tilde{B}_{m-1}(x)+\cdots+a_0\beta_0\tilde{B}_0(x).
\eeq
Equating coefficients from above, yields, $a_{m}\beta_{m-1}=a_{m-1}\beta_{m-1}$, $a_{m}\beta_{m-2}=a_{m-2}\beta_{m-2}$, $\ldots$, and $a_{m}\beta_{0}=a_{0}\beta_{0}$. By assumption, $a_m\not= a_k$ for $0\le k<m$, thus $\beta_k=0$ for $0\le k < m$. Hence, we have $B_m(x)=\beta_m \tilde{B}_m(x)$ ($\beta_m\not=0$) as desired.  
\end{proof}

\begin{exmp}\label{ex:notsame}
In some sense, Theorem \ref{unique} is best possible. Consider the following diagonal differential operators, 
\beq
T[x^n]:=\lr{\sum_{k=0}^\infty Q_k(x) D^k}x^n=(-1)^nx^n, 
\eeq
and
\beq
\tilde{T}[H_n(x)]:=\lr{\sum_{k=0}^\infty \tilde{Q}_k(x) D^k}H_n(x)=(-1)^nH_n(x). 
\eeq
Using the recursive formula of Theorem \ref{piotr-thm1}, by induction, for every $k\in\N_0$, 
\beq
Q_k(x)=\tilde{Q}_k(x)=\frac{(-2)^k}{k!}x^k. 
\eeq
Hence, by linearity, $T=\tilde{T}$. Thus, $T$ is a diagonal differential hyperbolicity preserving operator that can diagonalize with two distinct non-trivial bases (the bases are not simply multiples of each other) (cf. Theorem \ref{unique}). 
\end{exmp}

\hspace{.5in}

\section{Hyperbolicity Preserving Diagonal Differential Operators}

We will now refine the results of Section \ref{section1}, specifically in the case of hyperbolicity preserving operators. 

\begin{defn}\label{defhp}
Let $T:\R[x]\to\R[x]$ be a linear operator on real polynomials. We say that $T$ is \textit{hyperbolicity preserving} if $T[p(x)]$ has only real zeros whenever $p(x)$ has only real zeros. Furthermore, if $T$ is a hyperbolicity preserving diagonal differential operator, 
\beq
T[B_n(x)]:=\left(\sum_{k=0}^\infty Q_k(x)D^k\right)B_n(x)=a_nB_n(x),\ \ \ \ \ n\in\N_0, 
\eeq
then the eigenvalue sequence, $\{a_n\}_{n=0}^\infty$, is referred to as a \textit{multiplier sequence} for the basis, $\{B_n(x)\}_{n=0}^\infty$, or simply a $\{B_n(x)\}_{n=0}^\infty$ \textit{multiplier sequence}. If for each $n\in\N_0$, $B_n(x)=x^n$, then $\{a_n\}_{n=0}^\infty$ is called a \textit{classical multiplier sequence}. 
\end{defn}

Multiplier sequences are well studied throughout the literature. We summarize below some of the important properties that will be of use in the sequel (see also \cite{CC89}, \cite[p. 341]{Lev80}, \cite[Proposition 45, p. 43]{Pio07}). 

\begin{thm}\label{thm:mslist}
Let $\{\gamma_k\}_{k=0}^\infty$ be a classical multiplier sequence. Then each of the following assertions hold: 
\begin{enumerate}
\item[\textup{(1)}] $\{|\gamma_k|\}_{k=0}^\infty$, $\{-|\gamma_k|\}_{k=0}^\infty$, $\{(-1)^k|\gamma_k|\}_{k=0}^\infty$, and $\{(-1)^{k+1}|\gamma_k|\}_{k=0}^\infty$ are also classical multiplier sequences, furthermore one of these is the original sequence, $\{\gamma_k\}_{k=0}^\infty$, and
\item[\textup{(2)}] the Tur\'an inequalities are satisfied; i.e., $\gamma_{k+1}^2-\gamma_{k}\gamma_{k+2}\ge 0$ for every $k\ge 0$. 
\end{enumerate}
\end{thm}

\begin{rmk}\label{remark}
The Tur\'an inequalities are also satisfied for the coefficients of polynomials with only real zeros (see \cite[Corollary 4.11, p. 108]{Fis06}). In particular, we note that if a polynomial has only real negative zeros then none of the polynomial coefficients are zero. Likewise, if a transcendental entire function can be approximated by polynomials with only real negative zeros and has no zeros at the origin, then none of the Taylor coefficients are zero. 
\end{rmk}

Another important result, in the study of multiplier sequences and hyperbolicity preserving operators, is the multiplier sequence categorization theorem of A. Piotrowski. 

\begin{thm}[{\cite[Theorem 158, p. 145]{Pio07}}]\label{bnmstocms}
If $\{a_n\}_{n=0}^\infty\subseteq\R$ is a $\{B_n(x)\}_{n=0}^\infty$ multiplier sequence \textup{(}$\deg(B_n(x))=n$, $B_0(x)\not\equiv 0$\textup{)}, then $\{a_n\}_{n=0}^\infty$ is also a classical multiplier sequence. 
\end{thm}

With the results in the previous section, we are now in a position to extend A. Piotrowski's result (Theorem \ref{bnmstocms}) and provide the exact relationship of the differential representations. 

\begin{thm}
Suppose $T$ is a diagonal differential operator with respect to $\{B_n(x)\}_{n=0}^\infty$ that is also hyperbolicity preserving, 
\beq
T[B_n(x)]:=\left(\sum_{k=0}^\infty Q_k(x)D^k\right)B_n(x)=a_nB_n(x),\ \ \ \ \ n\in\N_0.
\eeq
Then 
\beq
\tilde{T}[x^n]:=\left(\sum_{k=0}^\infty \frac{Q_k^{(k)}}{k!}x^kD^k\right)x^n=a_nx^n,\ \ \ \ \ n\in\N_0, 
\eeq
is a diagonal differential operator with respect to $\{x^n\}_{n=0}^\infty$ and is also hyperbolicity preserving. 
\end{thm}

\begin{proof}
By Theorem \ref{bnmstocms}, it is already know that if $T[B_n(x)]:=a_nB_n(x)$ is hyperbolicity preserving, then $\tilde{T}[x^n]:=a_n x^n$ must also be hyperbolicity preserving. Thus, the new aspect of this theorem is the discovery of the exact differential form of $\tilde{T}$. That is, we wish to show, 
\beq
\tilde{T}=\sum_{k=0}^\infty \frac{Q_k^{(k)}}{k!}x^kD^k. 
\eeq
Define, 
\beq
W:=\sum_{k=0}^\infty \frac{Q_k^{(k)}}{k!}x^kD^k. 
\eeq
Since $T[B_n(x)]=a_nB_n(x)$ for every $n\in\N_0$, then by Theorem \ref{thm:dia2}, for each $n\in\N_0$,
\beq
a_n=\sum_{k=0}^n \binom{n}{k} Q_k^{(k)}. 
\eeq
Thus, 
\beq
W[x^n]=\left(\sum_{k=0}^\infty \frac{Q_k^{(k)}}{k!}x^kD^k\right)x^n = \sum_{k=0}^n \frac{Q_k^{(k)}}{k!}x^k \frac{n!}{(n-k)!} x^{n-k} = \left(\sum_{k=0}^n \binom{n}{k} Q_k^{(k)}\right)x^n = a_n x^n. 
\eeq
Hence, $W$ and $\tilde{T}$ agree on the basis $\{x^n\}_{n=0}^\infty$. So, by linearity, $W=\tilde{T}$. 
\end{proof}

Similar to Section \ref{section1}, we now begin some work on the uniqueness of the differential representation for hyperbolicity preserving operators. In Theorem \ref{unique}, it is realized that non-repeating eigenvalues (i.e., $a_n\not=a_m$ for $n\not=m$) gives us unique eigenvectors (up to a multiple). In particular, strictly increasing eigenvalues must have unique eigenvectors (up to a multiple). Hence, the following case of hyperbolicity preservers is of interest. 

\begin{thm}
Suppose $T$ is a hyperbolicity preserving diagonal differential operator, 
\beq
T[B_n(x)]=\left(\sum_{k=0}^\infty Q_k(x)D^k\right)B_n(x) = a_n B_n(x),\ \ \ \ \ n\in\N_0,
\eeq
where $\{a_n\}_{n=0}^\infty$, $a_0>0$, is interpolatable by a non-constant polynomial. Then $T$ must have strictly increasing eigenvalues. 
\end{thm}

\begin{proof}
Multiplier sequences are either all non-negative, non-positive, or alternate in sign (Theorem \ref{thm:mslist} item (1)). Since $a_0>0$ and $\{a_n\}_{n=0}^\infty$ is interpolated by a non-constant polynomial (i.e., $a_n\to\infty$ or $a_n\to-\infty$), then $a_n\ge 0$ for all $n\in\N_0$. Furthermore, if $a_k\ge a_{k+1}$ for some $k$, then by the Tur\'an inequalities (Theorem \ref{thm:mslist} item (2)), $a_{n+1}\ge a_{n+2}$ for all $n\ge k$. Thus, since $\{a_n\}_{n=0}^\infty$ is interpolated by a non-constant polynomial (again, either $a_n\to\infty$ or $a_n\to-\infty$), then we must have $a_k<a_{k+1}$ for every $k\in\N_0$. 
\end{proof}

An analogous result can be formulated for operators that diagonalize on a basis, $\{B_n(x)\}_{n=0}^\infty$, where for each $n\in\N_0$, $B_n(x)$ has real simple zeros (see \cite{FTW11}). We also note that all finite order diagonal differential operators have polynomial interpolated eigenvalues (Theorem \ref{cor:intercon}). Thus, we find that many hyperbolicity preserving operators have a unique diagonalization (up to a multiple, see Theorem \ref{unique}). Of course, it should be noted that Example \ref{ex:notsame} serves to show that there are hyperbolicity preserving operators that do not have unique diagonalizations. 

Increasing eigenvalues actually provide even further restrictions on the differential representation of hyperbolicity preserving diagonal differential operators (cf. Theorem \ref{thm:degqk}). 

\begin{thm}\label{final}
Suppose $T$ is a hyperbolicity preserving diagonal differential operator, 
\beq
T[B_n(x)]=\left(\sum_{k=0}^\infty Q_k(x)D^k\right)B_n(x) = a_n B_n(x),\ \ \ \ \ n\in\N_0, 
\eeq
where $0<a_n<a_{n+1}$ for each $n\in\N_0$. If $\{a_n\}_{n=0}^\infty$ is interpolated by a polynomial, $p(x)$, $\deg(p(x))=m$, then $\deg(Q_k(x))=k$ for all $0\le k\le m$. Likewise, if $\{a_n\}_{n=0}^\infty$ cannot be interpolated by a polynomial, then $\deg(Q_k(x))=k$ for every $k\in\N_0$. 
\end{thm}

\begin{proof}
By Theorem \ref{bnmstocms}, $\{a_n\}_{n=0}^\infty$ is a classical multiplier sequence. Since $\{a_n\}_{n=0}^\infty$ is increasing, then by T. Craven and G. Csordas \cite{CC83}, 
\beq
f(x):=e^{-x}\left(\sum_{k=0}^\infty \frac{a_k}{k!}x^k\right)=\sum_{n=0}^\infty \frac{1}{n!}\left(\sum_{k=0}^n \binom{n}{k} a_k (-1)^{n-k}\right)x^n = \sum_{k=0}^\infty \frac{Q_k^{(k)}}{k!} x^k 
\eeq
is approximable by polynomials with only real negative zeros (see also \cite{PS14}). If $\{a_n\}_{n=0}^\infty$ is interpolated by a polynomial, $p(x)$, $\deg(p(x))=m$, then by Theorem \ref{thm:degqk}, $Q_m^{(m)}\not=0$ and $Q_k^{(k)}=0$ for $k>m$. Hence, $f(x)$ is a polynomial of degree $m$ with only real negative zeros. Therefore, by Remark \ref{remark}, $Q_k^{(k)}\not=0$ for each $0\le k\le m$. Similarly, if $\{a_n\}_{n=0}^\infty$ is not interpolatable by a polynomial, then by Theorem \ref{thm:degqk}, $Q_k^{(k)}\not=0$ for infinitely many $k$. Hence, $f(x)$ is a transcendental entire function that is approximable by polynomials with only real negative zeros (furthermore, since $a_0\not=0$, $Q_0^{(0)}\not=0$, so, $f(x)$ has no zeros at the origin). Therefore, by Remark \ref{remark}, $Q_k^{(k)}\not=0$ for all $k\in\N_0$. 
\end{proof}

\hspace{.5in}

\section{Acknowledgments}

The author would like to formally thank the members of Dr. Csordas' Functional Analysis seminar, where many fruitful discussions and arguments took place in the preparation of this paper. In particular, thanks is extended to Dr. Forg\'acs and Mr. Hallum for listening to hours of the author's mathematical rants. Special thanks goes to Dr. Csordas, whose advice, critiques, guidance, and criticisms are greatly appreciated. 

\hspace{1in}

\printbibliography

\end{document}